\documentclass[reqno]{amsart}
\usepackage{amssymb}
\usepackage{amscd,amsfonts}
\usepackage{tikz}
\usepackage{pgfplots}
\usetikzlibrary{automata,positioning,calc,trees}
\usetikzlibrary{intersections,pgfplots.fillbetween}
\usetikzlibrary{shapes.geometric, arrows}
\usepackage{graphicx,tikz}
\usepackage{pgf,tikz,pgfplots}
\usepackage{mathrsfs}
\usepackage{enumerate}
\usepackage[shortlabels]{enumitem}
\usepackage{mathrsfs}
\usepackage{amssymb,amsmath,amsthm,color}
\usepackage{caption,subcaption}
\usepackage{hyperref}
\usepackage{cleveref}
\usepackage[alphabetic,bibtex-style]{amsrefs}
\usepackage{url}
\usepackage{setspace}
\usepackage{float}
\usepackage{makecell}

\BibSpec{article}{%
	+{}  {\PrintAuthors}                {author}
	+{,} { \textit}                     {title}
	+{.} { }                            {part}
	+{:} { \textit}                     {subtitle}
	+{,} { \PrintContributions}         {contribution}
	+{.} { \PrintPartials}              {partial}
	+{,} { }                            {journal}
	+{}  { \textbf}                     {volume}
	+{}  { \PrintDatePV}                {date}
	+{,} { \issuetext}                  {number}
	+{,} { \eprintpages}                {pages}
	+{,} { }                            {status}
	+{,} { \url}                        {url}    
	+{,} { \PrintDOI}                   {doi}
	+{,} { available at \eprint}        {eprint}
	+{}  { \parenthesize}               {language}
	+{}  { \PrintTranslation}           {translation}
	+{;} { \PrintReprint}               {reprint}
	+{.} { }                            {note}
	+{.} {}                             {transition}
}
\BibSpec{book}{%
	+{}  {\PrintAuthors}                {author}
	+{,} { \textit}                     {title}
	+{.} { }                            {part}
	+{:} { \textit}                     {subtitle}
	+{,} { \PrintContributions}         {contribution}
	+{.} { \PrintPartials}              {partial}
	+{,} { }                            {series}
	+{}  { \textbf}                     {volume}
	+{}  { \PrintDatePV}                {date}
	+{,} { \issuetext}                  {number}
	+{,} { \eprintpages}                {pages}
	+{,} { }                            {status}
	+{,} { \url}                        {url}    
	+{,} { available at \eprint}        {eprint}
	+{}  { \parenthesize}               {language}
	+{}  { \PrintTranslation}           {translation}
	+{;} { \PrintReprint}               {reprint}
	+{.} { }                            {note}
	+{.} {}                             {transition}
	+{,} { \PrintDOI}                   {doi}
}
\newcommand{\R}{{\mathbb {R}}}

\newcommand{\N}{{\mathbb N}}
\newcommand{\Z}{{\mathbb Z}}
\newcommand{\C}{{\mathbb C}}

\newcommand{\T}{{\mathbb T}}

\newtheorem{theorem}{Theorem}[section]
\newtheorem{lemma}{Lemma}[section]
\newtheorem{proposition}{Proposition}[section]
\theoremstyle{definition}

\theoremstyle{remark}

\numberwithin{equation}{section}

\title[] {Variational estimate for the family of discrete averages associated to simplices}

\author[Samanta]{Siddhartha Samanta}
\address{Siddhartha Samanta}
\email{siddhartha21@iiserb.ac.in}

 \address{Department of Mathematics\\
 	Indian Institute of Science Education and Research Bhopal\\
 	Bhopal-462066, India.}


\begin{document}
\begin{abstract}
	We prove $\ell^2(\Z^n)-$estimate of long $r$-variational seminorm for the family of discrete averages associated to simplices.
\end{abstract}
\maketitle
\section{Introduction}\label{sec:intro} 

	\subsection{Discrete spherical averages}For a function $f:\mathbb{Z}^{n} \rightarrow \mathbb{C}$ in $\ell^{p}(\mathbb{Z}^{n})$, with $1\leq p\leq \infty$ and $\lambda \geq 1 $, the discrete spherical average is defined by
	\[M_{\lambda}f(x) = \frac{1}{r_n(\lambda)}\sum_{|y|^2 = \lambda}f(x-y),\]
	where $y$ is restricted to range over integer lattice points. Here $r_n(\lambda)$ is the number of integers $m\in\Z^n$ such that $|m|^2 = \lambda$. Notice that $r_n(\lambda)=N_n(\lambda)$, where $N_n(\lambda)$ is the standard function that counts the number of ways to represent $\lambda$ as a sum of $n$ squares. For $n\geq 5$, we have $c_{n}^{1}\lambda^{\frac{n}{2}-1}\leq r_n(\lambda)\leq c_{n}^{2}\lambda^{\frac{n}{2}-1}$ for some constants $0<c_n^1< c_n^2 <\infty$.
	In $2002$, Magyar-Stein-Wainger \cite{magyarannals} proved that the full discrete spherical maximal function 
	$M_{*}f(x) = \sup_{ \lambda \geq 1 }|M_{\lambda}f(x)|$
	is bounded on $\ell^p(\Z^n)$ whenever $p>\frac{n}{n-2}$ and $n\geq 5$. Later, Ionescu \cite{Ionescu} established the restricted weak-type estimate for the operator $M_{*}$ at the endpoint $\frac{n}{n-2},\;n\geq 5$.  Note that the ranges of $p$ and $n$ are best possible in the above results.  We call a sequence $\{\lambda_l\}_{l=1}^{\infty}=\mathbb{L}$ lacunary if for all $l\geq 1$ we have $\frac{\lambda_{l+1}}{\lambda_l}>c$, where $c$ is a real number strictly greater than $1$. The lacunary discrete spherical maximal function associated to any lacunary sequence $\{\lambda_l\}_{l=1}^{\infty}$ is defined as
	 \[M_{lac}f(x) = \sup\limits_{ l \geq 1 }|M_{\lambda_{l}}f(x)|,\]    
	We remark that, in contrast to the continuous case, the discrete lacunary spherical maximal operator $M_{lac}$ does not admit boundedness on $\ell^p(\Z^n)$ for values of $p$ below $\frac{n}{n-1}$. An explicit construction of the counterexample is given in \cite{Cook&Hughes1}*{section 7}. Kesler-Lacey-Mena \cite{Lacey3} proved that the operator $M_{lac}$ is bounded on $\ell^p(\Z^n)$ for the range $p>\frac{n-2}{n-3}$ and $n\geq 5$. Nevertheless, the question concerning the $\ell^p(\Z^n)$-boundedness of the operator $M_{lac}$ remains unresolved for values of $p$ between $\frac{n}{n-1}$ and $\frac{n-2}{n-3}$.

	\subsection{Averages associated to simplices in $\Z^n$} We call a finite set $R$ Ramsey if for every number of colors $r\in\N$, there is a dimension $n=n(R,r)$, for which every $r$-coloring of $\R^n$ contains a monochromatic congruent copy of $R$. The fundamental question in geometric Ramsey theory, originated in the work of Erd\H{o}s et al. \cite{ErdosRamseyTheoremI}, is to characterize the finite sets $R$ that are Ramsey. In particular, it is conjectured that a finite set $R$ is Ramsey if and only if $R$ is a subset of a sphere. 
	From now on, for any $1\leq k\leq n$, we call a configuration $S=\{0, s_1,\dots,s_k\}\subset\Z^n$ as a non-degenerate $k$-simplex if the vectors $s_1,\dots,s_k$ are linearly independent. Non-degenerate simplices constitute one of the canonical examples of Ramsey sets \cite{NondegenerateSimplicesAsRamsey}. For a non-degenerate $k$-simplex $S$ and $\lambda\in\sqrt{\N}$, we say that a simplex $\bar{S}=\{0, m_1,\dots, m_k\}$ is isometric to $S$, denoted $\bar{S}\simeq\lambda S$, if $|s_i-s_j|=\lambda |m_i-m_j|$ for all $0\leq i,j\leq k.$ Now, we define \[\mathcal{S}_{\lambda S}:=\{(m_1,\dots, m_k)\in\Z^{nk}: \bar{S}=\{0, m_1,\dots, m_k\}\simeq\lambda S\}.\]
	Let $|\mathcal{S}_{\lambda S}|$ denote the number of isometric copies of $\lambda S$. It is well-known from \cites{SimplicesMagyar,DukeMagyar, SimplicesCookmultilinear} that, for $n\geq 2k+3$, the quantity $|\mathcal{S}_{\lambda S}|$ satisfies the bounds $\mathfrak{c}_{S}^1\lambda^{nk-k(k+1)}\leq|\mathcal{S}_{\lambda S}|\leq\mathfrak{c}_{S}^2\lambda^{nk-k(k+1)}$ for some constants $0<\mathfrak{c}_{S}^1<\mathfrak{c}_{S}^2<\infty$; i.e., $|\mathcal{S}_{\lambda S}|\approx\lambda^{nk-k(k+1)}$. Next, observe that for $k=1$ and $s_1=(1,0,\dots,0)$, the set $\mathcal{S}_{\lambda S}$ is precisely the set of integer lattices over the surface of the Euclidean sphere of radius $\lambda$. Since discrete spherical averages of functions have been extensively studied, it is interesting to investigate the behaviour of the functions when sampled along isometric copies of a non-degenerate $k$-simplex. Thus, it is
	natural to ask about the $\ell^p$-bounds for the maximal function associated to the discrete simplical averages defined as follows.
	 For a function $f:\Z^{nk}\rightarrow \R$, we define the discrete linear simplical average as 
	\[A_{\lambda S}f(x)=\frac{1}{|\mathcal{S}_{\lambda S}|}\sum_{y\in\mathcal{S}_{\lambda S}}f(x-y).\]
	Lyall, Magyar, Newman, and Woolfitt \cite{SimplicesMagyar} prove that the maximal function associated to the discrete simplical averages
	$A_{*}f(x):=\sup_{\lambda\geq 1}|A_{\lambda S}f(x)|$
	is bounded on $\ell^2(\Z^{nk})$ for all $k\geq 1,\; n\geq 2k+3$, and a non-degenerate $k$-simplex $S\subseteq\Z^n$. However, the question regarding the $\ell^p$-boundedness of the maximal function $A_{*}$ for values of $p$ other than $2$ remains open. 

	\subsection{The $r$-variation for the family of discrete averages} Let $U\subset\N$, and $f:\Z^n\rightarrow\C$ be a function in $\ell^p(\Z^n)$ with $1\leq p\leq\infty$. Let $r\in [1,\infty]$ and $\mathcal{T}=\{T_{\lambda}:\lambda\in U\}$ be a family of operators acting on the function $f$. The $r$-variation corresponding to the family $\mathcal{T}f=\{T_{\lambda}f:\lambda\in U\}$ is defined by
	\begin{align}
		V_{r}(\mathcal{T}f)(x) := \begin{cases} 							\sup\limits_{L\in\N}										\sup\limits_{\lambda_1<\dots<\lambda_L}\left(\sum\limits_{l=1}^{L-1}|T_{\lambda_{l+1}}f(x)-T_{\lambda_{l}}f(x)|^r\right)^{\frac{1}{r}}, & 0<r<\infty, \\
	\sup\limits_{\lambda_0<\lambda_1}\left|T_{\lambda_{1}}f(x)-				T_{\lambda_{0}}f(x)\right|, & r=\infty,\end{cases}
	\label{definition of variation}
	\end{align}
   
    where the supremum is taken over all finite increasing sequences $\{\lambda_1,\dots,\lambda_L\}$ lying in $U$. Since the beginning of the study of $r$-variations by L\'{e}pingle \cite{variationformartingales}, intensive research has been devoted to the study of $r$-variation across various classes of discrete and continuous averages. Contributions by numerous mathematicians include, for instance, \cites{DimensionfreeestimatesdiscreteHL,Bootstrap,variationalestimatesradon,Variationalongprimesandpolynomials, variationalinequalitycontinuous} among others; see also the references therein for related developments. In particular, Jones, Seeger, and Wright \cite{variationalinequalitycontinuous} studied $L^p-$estimates of the $r$-variation for continuous spherical averages. In the article \cite{variationBirchMagyar}, we initiate the study of variational estimates for the class of discrete averages over certain hyper-surfaces, which includes the sphere. Note that, we define $V_{r}(\mathcal{T}f)$ as the long $r$-variation corresponding to the family $\mathcal{T}f$, obtained by restricting the supremum in \eqref{definition of variation} to a lacunary sequence. As working with any lacunary sequence is equivalent to working with the sequence $\{2^l\}_{l=1}^\infty$, now onwards we will stick to the lacunary sequence $\mathbb{L}=\{2^l\}_{l=1}^\infty$.
	In this article, we study $\ell^{2}(\mathbb{Z}^{n})$-boundedness of $r$-variation of the family of discrete simplical averages $\mathcal{A}_\mathbb{L}^{S}=\{A_{\lambda S}:\;\lambda\in\mathbb{L}\}$. 
	In this article, we study $\ell^{2}(\mathbb{Z}^{n})$-boundedness of the long $r$-variations of the family of discrete averages associated to simplices. Let $\mathcal{A}_\mathbb{L}^{S}=\{A_{\lambda S}:\;\lambda\in\mathbb{L}\}$ be the family of discrete simplical averages over lacunary sequence $\mathbb{L}$. Then, the long $r-$variations associated to the family $\mathcal{A}_\mathbb{L}^{S}$ is defined as follows:
	\[V_r(\mathcal{A}_\mathbb{L}^S f)(x):=V_r(\{A_{\lambda S} f,\;\lambda\in\mathbb{L}\})(x).\]
    We note that
	\[|A_{lac}^S f(x)|\leq |A_{\lambda_0 S}f(x)|+2V_r(\mathcal{A}_\mathbb{L}^S f)(x),\]    
  	holds for any $\lambda_0\in\N$ and $1\leq r<\infty$. Thus, the $\ell^p-$ bounds for the long $r-$ variations imply the corresponding bounds for the lacunary maximal operator $A_{lac}^S$. Also observe that, the $\ell^p-$ boundedness of $V_r(\mathcal{A}_\mathbb{L}^S f)$ immediately implies that the term in \eqref{definition of variation} corresponding to the family $\mathcal{A}_\mathbb{L}^S f$ is finite for almost every $x\in\Z^{nk}$, which in turn implies almost everywhere convergence of the sequence $\{A_{\lambda S}f\}$ as $\lambda\rightarrow\infty$. Our main result of this paper is the following.
	\begin{theorem}\label{thm:varSimplex}
		Let $S=\{s_0=0,s_1,s_2,\dots,s_k\}\subseteq\Z^n$ be a non-degenerate $k$-simplex and $\mathbb{L}=\{2^l\}_{l=1}^\infty$, then for $r>2$, $k\geq 1$,  and $n\geq 2k+3$ we have 
		\begin{equation}\label{mainineq:Vr}
			\|V_r(\mathcal{A}_{\mathbb{L}}^S f)\|_{\ell^2(\Z^{kn})}\lesssim \|f\|_{\ell^2(\Z^{kn})}.
		\end{equation}
	\end{theorem}
	 The proof of \Cref{thm:varSimplex} relies on the jump inequalities corresponding to the family $\mathcal{A}_\mathbb{L}^{S}$. Thus, we now define the $\lambda$-jump corresponding to the family $\mathcal{A}_\mathbb{L}^{S}$, as the supremum of all integers $M\in\N$ for which there is an increasing sequence $0<u_1<v_1\leq u_2<v_2\leq\cdots\leq u_M<v_M$ with $u_i,v_i\in\mathbb{L}$, such that
    \[|A_{u_l S}f(x)-A_{v_l S}f(x)|>\lambda\] 
   holds for each $l=1,\dots,M$. We denote the $\lambda$-jump corresponding to the family $\mathcal{A}_\mathbb{L}^{S}$ by $\mathcal{J}_{\lambda}(\mathcal{A}_\mathbb{L}^{S})$. We suggest the article \cite{variationalinequalitycontinuous} to the interested readers for further details involving $\lambda-$jumps. The following lemma is the discrete analogue of \cite{variationalinequalitycontinuous}*{Lemma 2.1}, which plays a crucial role in the proof of \Cref{thm:varSimplex}. Since its proof is identical to that of \cite{variationalinequalitycontinuous}*{Lemma 2.1}, we omit the details.

    \begin{lemma}\cite{variationalinequalitycontinuous}*{Lemma 2.1}\label	{thm:jump implies variation b}
    	Let $\mathcal{T}$ be a family of linear operators from $\ell^2(\Z^{kn})$ to the set of measurable functions on $\Z^{kn}$. Suppose
    	\begin{align*}
        	\sup_{\lambda>0}\|\lambda(\mathcal{J}_{\lambda}(\mathcal{T}f))^{1/2}\|_{\ell^2}&\lesssim\|f\|_{\ell^2}
    	\end{align*}
    	holds for all $f\in\ell^2(\Z^{kn})$. Then, for all $f\in\ell^2(\Z^{kn})$ and $r>2$ we have,
    	\begin{align*}
        	\|V_{r}(\mathcal{T}f)\|_{\ell^2}&\lesssim\|f\|_{\ell^2}.
    	\end{align*}
    \end{lemma}
	It thus suffices to establish the next result, from which \Cref{thm:varSimplex} will follow.
	
    \begin{theorem}\label{thm:jumpSimplex}
    	Let $f\in\ell^2(\Z^{kn})$, then we have
        \begin{align*}
            \sup_{\lambda>0}\|\lambda(\mathcal{J}_{\lambda}(\mathcal{A}_\mathbb{L}^{S}f))^{1/2}\|_{\ell^2}&\lesssim\|f\|_{\ell^2}.
        \end{align*}
    \end{theorem}
	To prove \Cref{thm:jumpSimplex}, we employ a key decomposition of the function via the discrete analogue of the Littlewood-Paley theory for $\ell^2$-functions along with the appropriate square function estimates of each term. In addition, the argument requires the $\ell^2$ boundedness of the local maximal function for the discrete simplical averages from \cite{SimplicesMagyar}. We will discuss these ingredients in detail in the subsequent sections.

	\subsection{Notations}
	We use the following notation throughout the paper.
	\begin{enumerate}[i)]
		\item The notation $\mathscr{F}_{\Z^n}$ and $\mathscr{F}_{\R^n}$ are used to denote the Fourier transform on $\Z^{n}$ and $\R^n$ respectively. 
		\item The notation $M\lesssim N$ (or $N\gtrsim M$) means that there exists an absolute constant $0<C<\infty$ independent of the parameters on which $M$ and $N$ depend, such that $M\leq CN$ (or $M\geq CN$). Also, $M\approx N$ means that $M\lesssim N$ and $M\gtrsim N$.
		\item We denote $e(t)=e^{2\pi it}$, for $t\in\R$.
		\item The group $\mathbb{Z}_{q}$ denotes $\mathbb{Z}/q\mathbb{Z}$, and $U_{q}$ denotes the multiplicative group $\mathbb{Z}_{q}^{\times}$ with the understanding that $U_{1}=\mathbb{Z}_{1}=\{0\}$. We also denote $\Z_{q}^{n}=\Z_{q}\times\Z_{q}\times\dots\times\Z_{q}~(\text{n-times})$.
		\item For a function $f$ defined on $\R^n$, and for any $a\in\R$, the dilation is defined by $f_{a}(x)=\frac{1}{a^n}f\left(\frac{x}{a}\right)$.
	\end{enumerate}


\section{Proof of \Cref{thm:jumpSimplex}}\label{sec:proofjumpBM}
        For any integer $j\geq 0$, we consider $t_{j}=\operatorname{lcm}\left\{1,2,3, \ldots, 2^{j}\right\} \approx e^{2^{j}}$. For non-negative integers $j$ and $l$ that satisfy $2^{j} \leq l$, let 
		\begin{align*}
			\Omega_{l, j}&:=\left\{\alpha \in \T^{kn}: \alpha \in\left[-2^{j-l}, 2^{j-l}\right]^{kn}+\left(t_{j}^{-1}\Z\right)^{kn}\right\}.
		\end{align*}
    	Let $\psi$ be a Schwartz function with

    $$
        \mathscr{F}_{\R^{kn}}(\psi)(\xi)=\left\{\begin{array}{ll}
        1 & \text { if } \xi \in Q, \\
        0 & \text { if } \xi \in(2Q)^{c},
        \end{array}\right. 
    $$
    where $Q=\left[-\frac{1}{2}, \frac{1}{2}\right]^{kn}.$
    
    For $l \in \N$ and $J>l,$ define $\psi_{l, J}: \Z^{kn} \rightarrow \R$ by
    $$
    	\psi_{l, J}(x)=\left\{\begin{array}{cc}
    	\left(\frac{l}{J}\right)^{kn} \psi\left(\frac{x}{J}\right) & \text { if } x \in(l \Z)^{kn}, \\
    	0 & \text { otherwise. }
    	\end{array}\right.
    $$
    For $l \in \N$ and $0 \leq j \leq J_{l}:=\log _{2}(l)$, we define $\Psi_{l,j}^s=\psi_{t_j, (2|s|)^{l-j}}$, and $\Delta \Psi_{ l,j}^s=\Psi_{l,j+1}^s-\Psi_{l,j}^s$, where $s$ is the pre-assigned integer lattice $(s_1,\dots,s_k)\in\Z^{kn}$. We will need the following property of $\Delta \Psi_{ l,j}^s$.

	\begin{proposition}\label{prop:Simplexl2bound square fn b}
		The estimate 
		\[\sum_{l\geq 2^j}|\mathscr{F}_{\Z^{kn}}(\triangle\Psi_{l,j}^s)(\xi)|^2\lesssim_{\Psi^s}1\] 
		holds uniformly in $j\in\N$ and $\xi\in\T^{kn}$.
	\end{proposition}

	Before we prove this proposition, we note that for any $l\in\N$ and $J>l$,
	\[|\mathscr{F}_{\Z^{kn}}(\psi_{l,J})(\xi)|\lesssim 1,\]
	which is obtained via the Poisson summation formula, with a detailed derivation provided in \cite{variationBirchMagyar}.
	

	\begin{proof}[\textbf{Proof of \Cref{prop:Simplexl2bound square fn b}:}]
		We prove this proposition in the spirit of \cite[Lemma 1]{Magyarl2bound}.
		For $l\geq 2^j$, define 
		\begin{align*}
			\Omega^{s}_{l,j}&:=\left\{\alpha \in \T^{kn}: \alpha \in\left[-(2|s|)^{j-l}, (2|s|)^{j-l}\right]^{kn}+\left(t_{j}^{-1}\Z\right)^{kn}\right\}.
		\end{align*}
		Since $\left[-(2|s|)^{j-(l+1)}, (2|s|)^{j-(l+1)}\right]^{kn}\subset \left[-(2|s|)^{j-l}, (2|s|)^{j-l}\right]^{kn}$, we have $\Omega^{s}_{l+1,j}\subseteq\Omega^{s}_{l,j}$. Thus, $\Omega^{s}_{l,j}\subseteq \Omega^{s}_{2^j,j}$ for all $l\geq 2^j$. Next, observe that
		\begin{align}
		\mathscr{F}_{\Z^{kn}}(\psi_{l,J})(\xi)
		&=\sum_{x\in \Z^{kn}}\left(\frac{l}{J}\right)^{kn}\psi\left(\frac{l}{J}x\right)e^{-2\pi ix\cdot l\xi}\nonumber\\
		&=\sum_{x\in \Z^{kn}}\psi_{J/l}(x)e^{-2\pi ix\cdot l\xi}\nonumber\\\
		&=\sum_{x\in \Z^{kn}}\mathscr{F}_{\R^{kn}}(\psi_{J/l})(x+l\xi)\nonumber\\
		&=\sum_{x\in \Z^{kn}}\mathscr{F}_{\R^{kn}}(\psi)\left(\frac{J}{l}(x+l\xi)\right).\label{bound for psi hat}
	\end{align}
		By the definition of $\psi$ along with \eqref{bound for psi hat} we obtain that, $\mathscr{F}_{\Z^{kn}}(\psi_{t_j, (2|s|)^{l-j}})(\xi)$ is nonzero if $\left|\frac{x}{t_j}+\xi\right|\leq (2|s|)^{j-l}$, for $x\in\Z^{kn}$; i.e., $\operatorname{supp}\left(\mathscr{F}_{\Z^{kn}}(\psi_{t_j, (2|s|)^{l-j}})\right)\subseteq\Omega^{s}_{l,j}$. Now,
		\begin{align*}
			&\mathscr{F}_{\Z^{kn}}(\triangle\Psi_{l,j}^s)(\xi)\\
			&=\mathscr{F}_{\Z^{kn}}(\Psi_{l+1,j}^s)(\xi)-\mathscr{F}_{\Z^{kn}}(\Psi_{l,j}^s)(\xi)\\
			&=\sum_{x\in \Z^{kn}}\left[\mathscr{F}_{\R^{kn}}(\psi)\left(\frac{(2|s|)^{l+1-j}}{t_j}(x+t_j\xi)\right)-\mathscr{F}_{\R^{kn}}(\psi)\left(\frac{(2|s|)^{l-j}}{t_j}(x+t_j\xi)\right)\right].
		\end{align*}
		Therefore, $\mathscr{F}_{\Z^{kn}}(\triangle\Psi_{l,j}^s)(\xi)$ is nonzero if $\frac{1}{2}(2|s|)^{j-l-1}< \left|\frac{x}{t_j}+\xi\right|\leq (2|s|)^{j-l}$. For any $\xi\in\Omega^{s}_{2^j,j}$, define $l'=l'(j):=\operatorname{max}\{l\geq 2^j:\xi\in\Omega^{s}_{l,j}\}$. Note that, there exists a unique $b\in\Z^{nk}$ satisfying $|\xi-\frac{b}{t_j}|\leq(2|s|)^{j-l'}$. Since $l'$ is the maximum such that $\xi\in\Omega^{s}_{l',j}$, we have $\xi\in\Omega^{s}_{l',j}\setminus\Omega^{s}_{l'+1,j}$; i.e., $(2|s|)^{j-l'-1}\leq |\xi-\frac{b}{t_j}|\leq(2|s|)^{j-l'}$. Thus, $\mathscr{F}_{\Z^{kn}}(\triangle\Psi_{l,j}^s)(\xi)=0$ for $l>l'$, whereas for $2^j\leq l\leq l'$ we have $\mathscr{F}_{\Z^{kn}}(\Psi_{l,j}^s)(\xi)=\mathscr{F}_{\R^{kn}}(\psi)\left((2|s|)^{l-j}(\xi-\frac{b}{t_j})\right)$. Therefore, we have
		\begin{align*}
			\left|\mathscr{F}_{\Z^{kn}}(\triangle\Psi_{l,j}^s)(\xi)\right|
			&=\left|\mathscr{F}_{\R^{kn}}(\psi)\left((2|s|)^{l+1-j}(\xi-\frac{b}{t_j})\right)-\mathscr{F}_{\R^{kn}}(\psi)\left((2|s|)^{l-j}(\xi-\frac{b}{t_j})\right)\right|\\
			&\leq\left|(2|s|)^{l-j}\left(\xi-\frac{b}{t_j}\right)\right|\left|\mathscr{F}_{\R^{kn}}(\psi)'\left(\theta\right)\right|
			\lesssim_{\Psi^s}(2|s|)^{l-j}\left|\xi-\frac{b}{t_j}\right|\\
			&\leq (2|s|)^{l-l'}.
		\end{align*}
		Hence, it follows, using the bound of $\mathscr{F}_{\Z^{kn}}(\triangle\Psi_{l,j}^s)(\xi)$ that 
		\begin{align*}
			\sum_{l\geq 2^j}\left|\mathscr{F}_{\Z^{kn}}(\triangle\Psi_{l,j}^s)(\xi)\right|^2
			&\lesssim_{\Psi^s}\sum_{l=1}^{l'}\frac{1}{(2|s|)^{2(l'-l)}}\\
			&\lesssim_{\Psi^s} 1.
		\end{align*}
	\end{proof}
    We decompose the function $f$ as $f=f_1+f_2+f_3$, where the functions $f_1,\; f_2,\; f_3$ are defined as follows:
    \begin{align*}
    	f_1=f * \Psi_{l, 0}^s, \; f_2=\sum\limits_{j=0}^{J_l-3} f * \Delta \Psi_{l, j}^s ,\; \text{and}\; f_3=\left(f-f * \Psi_{l, J_{l}-2}^s\right).
    \end{align*}
	
    Using the sub-additivity property of $\lambda$-jump, we write 
    \begin{align*}
        \|\lambda(\mathcal{J}_{\lambda}(\mathcal{A}_\mathbb{L}^{S}f))^{1/2}\|_{\ell^2}
        &\leq \|\lambda(\mathcal{J}_{\lambda}(\mathcal{A}_\mathbb{L}^{S}f_1))^{1/2}\|_{\ell^2}
        +\|\lambda(\mathcal{J}_{\lambda}(\mathcal{A}_\mathbb{L}^{S}f_2))^{1/2}\|_{\ell^2}\\
        &\hspace{1cm}+\|\lambda(\mathcal{J}_{\lambda}(\mathcal{A}_\mathbb{L}^{S}f_3))^{1/2}\|_{\ell^2}.
    \end{align*}
	Thus, our goal reduces to establishing the following estimates 
	\begin{align}
		\|\lambda(\mathcal{J}_{\lambda}(\mathcal{A}_\mathbb{L}^{S}f_i))^{1/2}\|_{\ell^2}
        &\lesssim \|f\|_{\ell^2}~~\label{lambda jump bound b}
	\end{align}
	for all $i=1,2,3$. To establish \eqref{lambda jump bound b}, we will use the following pointwise inequality:
	\begin{align}
		\lambda \left[\mathcal{J}_{\lambda}(\mathcal{A}_\mathbb{L}^{S}f)(x)\right]^{1/2}
		&\lesssim\left(\sum_{l=1}^{\infty}|A_{\lambda_l S}f(x)|^2\right)^{1/2}.\label{square fn domination b}
	\end{align}
	The following $\ell^2-$estimate plays a crucial role in the proof of estimate~\ref{lambda jump bound b} for $i=2,3.$

	\begin{proposition}\cite{SimplicesMagyar}\label{prop:local l2 estimate}
		Let $S=\{s_0=0,s_1,s_2,\dots,s_k\}\subseteq\Z^n$ be a non-degenerate $k$-simplex. Then for  $k\geq 1$, $n\geq 2k+3$ with $1\leq j\leq J_l-2,$ and $l\in\N$ we have the estimate 
		\begin{align*}
			\left\|\sup_{2^l\leq\lambda\leq 2^{l+1}}|A_{\lambda S}f|\right\|_{\ell^2(\Z^{kn})}\lesssim_{n,\Psi^s} 2^{-j/2}j^{-1}\|f\|_{\ell^2(\Z^{kn})}
		\end{align*}
			holds whenever $\operatorname{supp}\mathscr{F}_{\Z^{kn}}(f)\subseteq(\Omega_{l,j}^s)^c$.
	\end{proposition}
	Although the original statement of the result in \cite{SimplicesMagyar} was proven under the assumption that $\operatorname{supp}\mathscr{F}_{\Z^{kn}}(f)\subseteq(\Omega_{l,j})^c$, the same argument applies when $\operatorname{supp}\mathscr{F}_{\Z^{kn}}(f)\subseteq(\Omega_{l,j}^s)^c$ as well. Next, we show that support of $\mathscr{F}_{\Z^{kn}}(f_2)=\mathscr{F}_{\Z^{kn}} \left(\sum_{j=0}^{J_l-3} f * \Delta \Psi_{l, j}^s\right) $ and $\mathscr{F}_{\Z^{kn}}(f_3)=\mathscr{F}_{\Z^{kn}}\left(f-f * \Psi_{l, J_{l}-2}^s\right)$ are contained in $(\Omega_{l,j}^s)^c$. This will allow us to utilize \Cref{prop:local l2 estimate} while proving the inequality \eqref{lambda jump bound b}.\\ To see this, let us observe that
	\begin{align*}
		&\mathscr{F}_{\Z^{kn}}(f_2)(\xi)\\
		&=\sum_{j=0}^{J_l-3} \mathscr{F}_{\Z^{kn}}(f)(\xi) \left[ \mathscr{F}_{\Z^{kn}}(\Psi_{l, j+1}^s)(\xi)- \mathscr{F}_{\Z^{kn}}(\Psi_{l, j}^s)(\xi) \right]\\
		&= \mathscr{F}_{\Z^{kn}}(f)(\xi)\sum_{j=0}^{J_l-3} \sum_{x\in \Z^{kn}}\left[\mathscr{F}_{\R^{kn}}(\psi)\left(\frac{(2|s|)^{l+1-j}}{t_j}(x+t_j\xi)\right)-\mathscr{F}_{\R^{kn}}(\psi)\left(\frac{(2|s|)^{l-j}}{t_j}(x+t_j\xi)\right)\right].
	\end{align*}
	Since the quantity in the square bracket is nonzero if $\frac{1}{2}(2|s|)^{(j-1)-l}< \left|\frac{x}{t_j}+\xi\right|\leq (2|s|)^{j-l}$; for each $\xi$, we have $\operatorname{supp}\mathscr{F}_{\Z^{kn}}(f_2)(\xi)\subset(\Omega^{s}_{l,j-2})^c$, for some $j\in [0,J_l]$. Similarly, we have
	\begin{align*}
		\mathscr{F}_{\Z^{kn}}(f_3)(\xi)
		&=\mathscr{F}_{\Z^{kn}}(f)(\xi)\left(1-\mathscr{F}_{\Z^{kn}}(\Psi_{l, J_{l}-2}^s)(\xi)\right)\\
		&=\mathscr{F}_{\Z^{kn}}(f)(\xi)\left(1-\sum_{x\in \Z^{kn}}\mathscr{F}_{\R^{kn}}(\psi)\left((2|s|)^{l-J_l+2}\left(\frac{x}{t_{J_l-2}}+\xi\right)\right)\right).
	\end{align*}
	Thus, $\operatorname{supp}\mathscr{F}_{\Z^{kn}}(f_3)\subset(\Omega^{s}_{l,J_l-3})^c$. Therefore, we have 
	$\operatorname{supp}\mathscr{F}_{\Z^{kn}}(f_2),\;\operatorname{supp}\mathscr{F}_{\Z^{kn}}(f_3)\subset(\Omega^{s}_{l,j})^c$, for some $j\in [0,J_l-2]$.

	\subsection{Proof of \eqref{lambda jump bound b} for $i=2$.} \label{proof of lambda jump for i=2}Notice that $\operatorname{supp}\mathscr{F}_{\Z^{kn}}(f_2)\subset\Omega^{c}_{j,l}$. Now, to prove this estimate first we use \cref{square fn domination b} to dominate the $\lambda$-jump by the corresponding square function, then to sum in $l$ we use the decay of $\ell^2$ bound of single average via \Cref{prop:local l2 estimate},
	\begin{align*}
		\|\lambda(\mathcal{J}_{\lambda}(\mathcal{A}_\mathbb{L}^{S}f_2))^{1/2}\|_{\ell^2}
		&\lesssim \left\|\left(\sum_{l=1}^{\infty}|A_{2^l S}f_2|^2\right)^{1/2}\right\|_{\ell^2}\\
		&\leq \left\|\left(\sum_{l=1}^{\infty}\left|\sum_{j=0}^{J_l-1}A_{2^l S}(f*\triangle\Psi_{l,j}^s)\right|^2\right)^{1/2}\right\|_{\ell^2}\\
		&\leq \sum_{j\geq 0}\left(\sum_{l\geq 2^j}\sum_{x\in\Z^{kn}}|A_{2^l S}(f*\triangle\Psi_{l,j}^s)(x)|^2\right)^{1/2}.
	\end{align*}
	We obtain the last inequality by using Minkowski's inequality and finally we invoke \Cref{prop:Simplexl2bound square fn b} to obtain the desired estimate as follows:
	\begin{align*}
		\|\lambda(\mathcal{J}_{\lambda}(\mathcal{A}_\mathbb{L}^{S}f_2))^{1/2}\|_{\ell^2}
		&\leq \sum_{j\geq 0}\left(\sum_{l\geq 2^j}\sum_{x\in\Z^{kn}}|A_{2^l S}(f*\triangle\Psi_{l,j}^s)(x)|^2\right)^{1/2}\\
		&\lesssim \sum_{j\geq 0}2^{-j/2}j^{-1}\left(\sum_{l\geq 2^j}\|f*\triangle\Psi_{l,j}^s\|_{\ell^2}^{2}\right)^{1/2}\\
		&\lesssim \|f\|_{\ell^2}\sum_{j\geq 0}2^{-j/2}j^{-1}\left(\sum_{l\geq2^j}\sup_{\xi\in\T^{kn}}|\mathscr{F}_{\Z^{kn}}(\triangle\Psi_{l,j}^s)(\xi)|^2\right)^{1/2}\\
		&\lesssim \|f\|_{\ell^2}.
	\end{align*}

	\subsection{Proof of \eqref{lambda jump bound b} for $i=3$.} \label{proof of lambda jump for i=3}Since $\operatorname{supp}\mathscr{F}_{\Z^{kn}}(f_3)\subset\Omega^{c}_{l,J_l-3}$, invoking \Cref{prop:local l2 estimate} for $j=J_l-3=(\log_2l)-3$ we get that 
	\begin{align*}
		\|\lambda(\mathcal{J}_{\lambda}(\mathcal{A}_\mathbb{L}^{S}f_3))^{1/2}\|_{\ell^2}
		&\lesssim \left\|\left(\sum_{l=1}^{\infty}|A_{2^l S}(f-f*\Psi_{l,J_l}^s)|^2\right)^{1/2}\right\|_{\ell^2}\\
		&= \left(\sum_{l=1}^{\infty}\left\|A_{2^l S}(f-f*\Psi_{l,J_l}^s)\right\|_{\ell^2}^{2}\right)^{1/2}\\
		&\lesssim \left(\sum_{l=1}^{\infty}2^{-\log_2l}(\log_2l)^{-2}\left(1+\|\mathscr{F}_{\Z^{kn}}(\Psi_{l,J_l}^s)\|_{\ell^{\infty}(\T^{kn})}\right)^2\|f\|_{\ell^2}^2\right)^{1/2}\\
		&\lesssim \|f\|_{\ell^2},
	\end{align*}
	where we use the estimate \eqref{bound for psi hat} to obtain the final step. At this point, it remains to establish the estimate \eqref{lambda jump bound b} for $i=1$, which is the major part of the proof. We will use the following decomposition of $\Z^{kn}$: For each non-negative integer $l$, we write
	\[\Z^{kn}=\bigcup\limits_{t\in\Z^{kn}}\mathcal{Q}_{t}^{l},\]
	where $\mathcal{Q}_{t}^{l}$ is the set of all lattice points lying in the dyadic cube $(2|s|)^l(t+[0,1)^{kn})$. Observe that, the collection $\{\mathcal{Q}_{t}^{l}:\;t\in\Z^{kn}\}$ forms a partition of $\Z^{kn}$ such that for $l_1\leq l_2$ we have either $\mathcal{Q}_{t}^{l_1}\subset \mathcal{Q}_{t}^{l_2}$ or $\mathcal{Q}_{t}^{l_1}\cap \mathcal{Q}_{t}^{l_2}=\emptyset$. Furthermore, for each $\mathcal{Q}_{t_1}^{l_1}$ with $l_1<l_2$, there exists a unique $t_2$ satisfying $\mathcal{Q}_{t_1}^{l_1}\subset \mathcal{Q}_{t_2}^{l_2}$.

	\subsection{Proof of \eqref{lambda jump bound b} for $i=1$.} \label{proof of lambda jump for i=1} We first decompose the average as follows: 
	\begin{align*}
		A_{2^{l} S}f_1&=(\Psi_{l,0}^s*w_{2^l S})*f-E_{l}f+E_{l}f, 
	\end{align*} 
	\sloppy where $E_{l}f=|\mathcal{Q}_{t}^{l}|^{-1}\sum_{y\in \mathcal{Q}_{t}^{l}}f(y)$, and $w_{\lambda S}(x)=|\mathcal{S}_{\lambda S}|^{-1}1_{\mathcal{S}_{\lambda S}}(x)$. Now onwards, we use the notation $\mathbb{E}f$ and $\mathcal{L}f$ to denote the sets $\{E_{l}f:l\in\N\}$ and $\left\{(\Psi_{l,0}^s*w_{2^l S})*f-E_{l}f:l\in\N\right\}$ respectively. Note the following subadditivity property of $\lambda$-jump:  
	\[\mathcal{J}_{\lambda}(\mathcal{A}_\mathbb{L}^Sf)\leq \mathcal{J}_{\lambda/2}(\mathcal{L}f)+\mathcal{J}_{\lambda/2}(\mathbb{E}f),\]

	which will be used in our proof. Let us recall the following $\ell^p-$estimate from \cite{variationalinequalitycontinuous}:
	\[\left\|\lambda \left(\mathcal{J}_{\lambda}(\mathbb{E}f)\right)^{1/2}\right\|_{\ell^{p}(\Z^n)}\lesssim \|f\|_{\ell^{p}(\Z^n)},~\;1<p<\infty.\] 
	Based on the above classical $\ell^p-$estimate for the jump corresponding to the family $\mathbb{E}f$ along with the estimate \eqref{square fn domination b}, the proof of \eqref{lambda jump bound b} for $i=1$ is reduced to showing the $\ell^2-$boundedness of the square function corresponding to the family $\mathcal{L}f$; i.e., it suffices to show that
	\begin{align}
		\left\|\left(\sum_{l=1}^{\infty}|(\Psi_{l,0}^s*w_{2^lS})*f-E_{l}f|^2\right)^{1/2}\right\|_{\ell^{2}}
	\lesssim \|f\|_{\ell^{2}}.\label{square fn bound}
	\end{align}
	To prove \eqref{square fn bound}, we rely on the following lemma, which is the discrete analogue of \cite[Lemma 3.2]{variationalinequalitycontinuous}. 
	\begin{lemma}\label{lem:main l2 estimate}
		Define $D_{m}f(x)=E_{m}f(x)-E_{m-1}f(x)$, then for any $f\in\ell^2(\Z^{kn})$ and $\delta>0$, the following estimate holds:
		\begin{align*}
			\|(\Psi_{l,0}^s*w_{2^lS})*D_{m}f-E_{l}(D_{m}f)\|_{\ell^2}
			&\lesssim 2^{-\delta|l-m|}\|D_{m}f\|_{\ell^2}.
		\end{align*}
	\end{lemma}
	We assume \Cref{lem:main l2 estimate} for the moment. Write $D_{m}f(x)=E_{m}f(x)-E_{m-1}f(x)$ and note that we have $f=\sum_{m\in\Z}D_{m}f$. Next, using this we obtain the following:
	\begin{align*}
		\left\|\left(\sum_{l=1}^{\infty}|(\Psi_{l,0}^s*w_{2^lS})*f-E_{l}f|^2\right)^{1/2}\right\|_{\ell^{2}}&=\left(\sum_{l=1}^{\infty}\|(\Psi_{l,0}^s*w_{2^lS})*f-E_{l}f\|_{\ell^2}^{2}\right)^{1/2}\\
		&\leq \left(\sum_{l=1}^{\infty}\left(\sum_{m\in\Z}\|(\Psi_{l,0}^s*w_{2^lS})*D_m f-E_{l}(D_m f)\|_{\ell^2}\right)^2\right)^{1/2}\\
		&\lesssim \left(\sum_{l=1}^{\infty}\left(\sum_{m\in\Z}2^{-\delta|l-m|}\|D_{m}f\|_{\ell^2}\right)^2\right)^{1/2}\\
		&\lesssim \left(\sum_{l=1}^{\infty}\sum_{m\in\Z}2^{-\delta|l-m|}\|D_{m}f\|_{\ell^2}^2\right)^{1/2}\\
		&\lesssim \left(\sum_{m\in\Z}\|D_{m}f\|_{\ell^2}^2\right)^{1/2}\\&\lesssim \|f\|_{\ell^2}.
	\end{align*}
This concludes the proof of \eqref{lambda jump bound b} for $i=1$. \qed

\subsection*{Proof of \Cref{lem:main l2 estimate}.} To establish this lemma we devide the whole analysis into two parts. More precisely, in first part we will discuss the case when $l\geq m$, and in the second part we will discuss the case when $l<m$.
\subsection*{Case 1.} Suppose $l\geq m$. Then we can easily see that the fact $E_{l}(E_	{m}f)=E_{l}f$ implies that $E_{l}(D_{m}f)=E_{l}(E_{m}f-E_{m-1}f)=0.$ Hence $\sum_{y\in \mathcal{Q}_{t}^{m}}D_{m}f=0$. Now, using this observation we can write 
	\begin{align}
		(\Psi_{l,0}^s*w_{2^lS})*D_{m}f(x)
		&=\sum_{t}\sum_{y\in \mathcal{Q}_{t}^{m}}\Psi_{l,0}^s*w_{2^lS}(x-y)D_{m}f(y)\nonumber\\
		&=\sum_{t}\sum_{y\in \mathcal{Q}_{t}^{m}}\left[\Psi_{l,0}^s*w_{2^lS}(x-y)-\Psi_{l,0}^s*w_{2^lS}(x-q_{t}^{m})\right]D_{m}f(y), \label{main l2 estimate eqn 1}
	\end{align}
	where $\mathfrak{q}_{t}^{m}$ is the center of the cube $\mathcal{Q}_{t}^{m}$. Further, by using the mean-value theorem, we obtain the following estimate for the term $|\Psi_{l,0}^s*w_{2^lS}(x-y)-\Psi_{l,0}^s*w_{2^lS}(x-q_{t}^{m})|$: 
	\begin{align*}
		&|\Psi_{l,0}^s*w_{2^lS}(x-y)-\Psi_{l,0}^s*w_{2^lS}(x-\mathfrak{q}_{t}^{m})|\\
		&\lesssim\frac{(2|s|)^{-lkn}}{2^{l(nk-k(k+1))}}\sum_{z\in\mathcal{S}_{2^lS}}\left|\psi\left(\frac{x-y-z}{(2|s|)^l}\right)-\psi\left(\frac{x-\mathfrak{q}_{t}^{m}-z}{(2|s|)^l}\right)\right|\\
		&\lesssim\frac{(2|s|)^{-lkn}}{2^{l(nk-k(k+1))}}\frac{|y-\mathfrak{q}_{t}^{m}|}{(2|s|)^l}\sum_{z\in\mathcal{S}_{2^lS}}\left|\tilde\psi\left(\frac{x-y-z}{(2|s|)^l}\right)\right|\\
		&\lesssim 2^{-(l-m)}\left|\tilde\psi_{(2|s|)^l}(x-y-z_0)\right|,
	\end{align*}
	where $\tilde\psi$ is a Schwartz function. Hence the above observation along with \eqref{main l2 estimate eqn 1} implies that
	\begin{align*}
		|(\Psi_{l,0}^s*w_{2^lS})*D_{m}f(x)|
		&\lesssim 2^{-(l-m)}\mathcal{M}(D_{m}f)(x-z_0),
	\end{align*} 
	here in the above expression $\mathcal{M}$ is the discrete analogue of the Hardy-Littlewood maximal function which is known to be bounded on $\ell^p(\Z^{kn}), \;\text{for}\; 1<p\leq \infty$. Consequently, in the case $l\geq m$, the required estimate
	\begin{align*}
		\|(\Psi_{l,0}^s*w_{2^lS})*D_{m}f-E_{l}(D_{m}f)\|_{\ell^2}
		&\lesssim 2^{-|l-m|}\|D_{m}f\|_{\ell^2}
	\end{align*}
	holds.
\subsection*{Case 2.} Let us consider $l< m.$ Note that, for $l<m$, we have $E_{l}(E_{m}f)=E_{m}f$. 
It follows, therefore, that $E_{l}(D_{m}f)=D_{m}f$. We employ the Poisson summation formula to deduce that
	\begin{align*}
		\sum_{x\in\Z^{kn}}\Psi_{l,0}^s*w_{2^lS}(x)
		&=\frac{1}{|\mathcal{S}_{2^lS}|}\sum_{x\in\Z^{kn}}\sum_{y\in\mathcal{S}_{2^lS}}\frac{1}{(2|s|)^{lkn}}\psi\left(\frac{x-y}{(2|s|)^l}\right)\\
		&=\frac{1}{|\mathcal{S}_{2^lS}|}\sum_{x\in\Z^{kn}}\sum_{y\in\mathcal{S}_{2^lS}}\mathscr{F}_{\R^{kn}}(\psi)\left((2|s|)^lx\right)e^{-2\pi ix\cdot y}=1.
	\end{align*}
	Accordingly, one can concludes the subsequent pointwise equality:
	\begin{align*}
		(\Psi_{l,0}^s*w_{2^lS})*D_{m}f(x)-E_{l}(D_{m}f)(x)
		&=\sum_{y\in\Z^{kn}}[D_{m}f(x-y)-D_{m}f(x)](\Psi_{l,0}^s*w_{2^lS})(y)\\
		&=\sum_{d\geq 0}T_{d}(x),
	\end{align*}
\sloppy where $T_d(x)=\sum\limits_{y\in E_{l,d}}[D_{m}f(x-y)-D_{m}f(x)](\Psi_{l,0}^s*w_{2^l,S})(y)$ and the set $E_{l,d}$ is defined as follows:
	$$
        E_{l,d}=\left\{\begin{array}{ll}
        \left\{y\in\Z^{kn}:(2|s|)^{l+d-1}\leq |y|\leq (2|s|)^{l+d}\right\} & \text { if } d>0, \\~\\
        \left\{y\in\Z^{kn}: |y|\leq (2|s|)^{l}\right\} & \text { if } d=0.
        \end{array}\right.
    $$
Now, we turn our attention to the $\ell^2(\Z^{kn})$ estimate of $T_d(x)$ for $d$ greater than $\epsilon |l-m|$, where $\epsilon$ is a suitably chosen positive real number. Indeed, we have
	\begin{align*}
		\|T_d\|_{\ell^2}
		&=\left(\sum_{x\in\Z^{kn}}\left|\sum\limits_{y\in E_{l,d}}[D_{m}f(x-y)-D_{m}f(x)](\Psi_{l,0}^s*w_{2^lS})(y)\right|^{2}\right)^{1/2}\\
		&\leq \sum_{y\in E_{l,d}}\left(\sum_{x\in\Z^{kn}}|D_{m}f(x-y)-D_{m}f(x)|^{2}|(\Psi_{l,0}^s*w_{2^lS})(y)|^{2}\right)^{1/2}\\
		&\lesssim\|D_mf\|_{\ell^2}\sum_{y\in E_{l,d}}|(\Psi_{l,0}^s*w_{2^lS})(y)|.
	\end{align*}
	Note that for $z=(z_1,\dots,z_k)\in\mathcal{S}_{2^lS}$, we have $|z_i-z_j|=2^l|s_i-s_j|$ for all $0\leq i,j\leq k$, with the convention that $z_0=0$. If we assume $j=0$ in the above equality, we get $|z_i|=2^l|s_i|$ for all $0\leq i\leq k$, thereby implying that $|z|=2^l|s|$ for all $z\in\mathcal{S}_{2^lS}$. Next, for $d>0$ and $y\in E_{l,d}$, we have that $(2|s|)^{l+d-1}\leq (2|s|)^l +|y|-|z|$ for $z\in\mathcal{S}_{2^lS}$. Therefore, for any large number $L\in\N$ we get that 
	\begin{align*}
		|(\Psi_{l,0}^s*w_{2^l,S})(y)|
		&=\frac{1}{2^{l(kn-k(k+1))}}\left|\sum_{z\in\mathcal{S}_{2^lS}}\psi_{1,(2|s|)^l}(y-z)\right|\\
		&\lesssim \frac{(2|s|)^{-lkn}}{2^{l(kn-k(k+1))}}\sum_{\substack{z\in\mathcal{S}_{2^lS},\\ y-z\in\Z^{kn}}}\left|\psi\left(\frac{y-z}{(2|s|)^l}\right)\right|\\
		&\lesssim \frac{(2|s|)^{-lkn}}{2^{l(kn-k(k+1))}}\sum_{\substack{z\in\mathcal{S}_{2^lS},\\ y-z\in\Z^{kn}}}\frac{(2|s|)^{lL}}{((2|s|)^l+|y|-|z|)^L}\\
		&\leq (2|s|)^{-L(d-1)}(2|s|)^{-lkn}.
	\end{align*}
	Hence, for $d\geq \epsilon |l-m|$, we have 
	\begin{align*}
		\sum_{d\geq \epsilon |l-m|}\|T_d\|_{\ell^2}
		&\lesssim \|D_mf\|_{\ell^2}\sum_{d\geq \epsilon |l-m|}\sum_{y\in E_{l,d}}(2|s|)^{-L(d-1)}(2|s|)^{-lkn}\\
		&=\|D_mf\|_{\ell^2}\sum_{d\geq \epsilon |l-m|}\frac{(2|s|)^{-l(kn-1)}}{(2|s|)^{(L-1)(d-1)}}\\
		&\lesssim_s \frac{2^{-(L-1)|l-m|\epsilon'}}{2^{l(kn-1)}}\|D_mf\|_{\ell^2}\\
		&\lesssim \frac{2^{-|l-m|\epsilon'}}{2^{l(kn-1)}}\|D_mf\|_{\ell^2}.
	\end{align*}
	Next, we address the remaining part of the sum, i.e., for $d\leq \epsilon |l-m|$. Observe that integers $x$ and $x-y$ lie in $\mathcal{Q}_{t}^{m-1}$ whenever $x\in\mathbb{A}:= \{z\in \mathcal{Q}_{t}^{m-1}:\text{dist}(z,\Z^{kn}\setminus \mathcal{Q}_{t}^{m-1})\geq (2|s|)^{d+m+(l-m)}\}$ and $|y|\leq (2|s|)^{d+l}$. Thus, for $x\in\mathbb{A}$ and $|y|\leq (2|s|)^{d+l}$, we have $D_{m}f(x-y)-D_{m}f(x)=0$. Therefore, to estimate the $\ell^2(\Z^{kn})-$norm of $T_d(x)$, we only require to consider those $x$ which belongs to the set $U$, defined as:
	\[U:=\{z\in \mathcal{Q}_{t}^{m-1}:\text{dist}(z,\Z^{kn}\setminus \mathcal{Q}_{t}^{m-1})\leq (2|s|)^{d+m+(l-m)}\}.\]
	 We note the following bound of the cardinality of $U$, which is needed in the subsequent analysis:
	 \[|U|\leq (2|s|)^{d\eta-\eta|l-m|}|\mathcal{Q}_{t}^{m-1}|\lesssim_s 2^{-\eta'|l-m|}|\mathcal{Q}_{t}^{m-1}|,\] 
	 for some $\eta\in(0,1)$. Using the decomposing of $\Z^{kn}$ into the cubes $\mathcal{Q}_{t}^{m-1}$, we get that
	\begin{align}
		\|T_d\|_{\ell^2}^2
		&= \sum_{t}\sum_{x\in \mathcal{Q}_{t}^{m-1}}|T_d(x)|^2\nonumber\\
		&\leq \sum_{t}\sum_{x\in U }|T_d(x)|^2\nonumber\\
		&\leq \sum_{t}\sup\limits_{x\in U}|T_d(x)|^2|U|\nonumber\\
		&\lesssim 2^{-\eta'|l-m|}\sum_{t}|\mathcal{Q}_{t}^{m-1}|\sup_{x\in \mathcal{Q}_{t}^{m-1}}|T_d(x)|^2.\label{l2 bound of Td}
	\end{align}
	Observe that, whenever $x\in \mathcal{Q}_{t}^{m-1}$ and $|y|\leq (2|s|)^{d+l}$, the integer $x-y$ belongs to the ball $B\left(\mathfrak{q}_{t}^{m-1}, C(2|s|)^{m}\right)\supseteq \mathcal{Q}_{t}^{m-1}$ for some constant $C>0$. By employing this, we obtain that 
	\begin{align}
		\sup\limits_{x\in \mathcal{Q}_{t}^{m-1}}|T_d(x)|^2\lesssim \sup\limits_{x\in B\left(\mathfrak{q}_{t}^{m-1}, C(2|s|)^{m}\right)}|D_mf(x)|^2. \label{pointwise bound of Td}
	\end{align}
	Define the set $V_{t,m}=\{h: \mathcal{Q}_{h}^{m-1}\cap B\left(\mathfrak{q}_{t}^{m-1}, C(2|s|)^{m}\right)\neq \phi \}$. Next, using the fact that, the number of integer lattice points in a ball of radius $C(2|s|)^m$ is of the order $(2|s|)^{mkn}$ one can clearly see that the cardinality of the set $V_{t,m}$ is uniformly bounded with respect to $t$. Moreover, for each $h$, the number of $t$ such that $h\in V_{t,m}$ is uniformly bounded. Finally, we invoke \eqref{pointwise bound of Td} in \eqref{l2 bound of Td} and incorporate the above observations to get the following:
	\begin{align*}
		\|T_d\|_{\ell^2}^2
		&\lesssim 2^{-\eta'|l-m|}\sum_{t}\sum_{h\in V_{t,m}}\sum_{x\in \mathcal{Q}_{h}^{m-1}}|D_m f(x)|^2\\
		&\lesssim 2^{-\eta'|l-m|}\sum_{h}\sum_{x\in \mathcal{Q}_{h}^{m-1}}|D_m f(x)|^2\\
		&=2^{-\eta'|l-m|}\|D_m f\|_{\ell^2}^2.
	\end{align*}
	Hence, to complete the proof of \Cref{lem:main l2 estimate}, it suffices to sum in $d$ for $d\leq \epsilon |l-m|$. Indeed, we have
	\begin{align*}
		\sum_{0\leq d\leq \epsilon |l-m|}\|T_d\|_{\ell^2}
		&\lesssim \|D_mf\|_{\ell^2}\sum_{0\leq d\leq \epsilon |l-m|}2^{-\eta'|l-m|}\\
		&\lesssim 2^{-\bar{\eta}|l-m|}\|D_mf\|_{\ell^2}.
	\end{align*}
	This concludes the proof.
 \qed

	\subsection*{Acknowledgement}
	The author gratefully acknowledges his PhD supervisor Prof. Saurabh Shrivastava, for suggesting this problem and for his careful reading of the article with valuable suggestions. The author is supported by IISER Bhopal for PhD fellowship.
	\bibliography{bibliography}
\end{document}